\documentclass[a4paper,12pt]{article}
\usepackage{amsmath,amsthm,amsfonts,amssymb,bbm}
\usepackage{graphicx,psfrag,subfigure}
\usepackage{graphics}
\usepackage{cite}
\usepackage{hyperref}
\usepackage{graphicx}

\hypersetup{colorlinks=true}

\renewcommand{\d}{\mathrm d}
\newcommand{\tac}{\mathrm{tac}}
\newcommand{\R}{\mathbb R}
\newcommand{\wt}{\widetilde}
\renewcommand{\Re}{\operatorname{Re}}
\newcommand{\Ai}{\operatorname{Ai}}

\newcommand{\A}{\mathcal{A}}
\newcommand{\B}{\mathcal{B}}
\newcommand{\C}{\mathcal{C}}
\newcommand{\D}{\mathcal{D}}
\renewcommand{\L}{\mathcal{L}}
\newcommand{\Id}{\mathbbm{1}}
\renewcommand{\O}{\mathcal{O}}
\renewcommand{\P}{\mathbf P}
\newcommand{\T}{\mathcal T}

\newcommand{\lto}{\longrightarrow}

\newtheorem{proposition}{Proposition}[section]
\newtheorem{theorem}[proposition]{Theorem}
\newtheorem{lemma}[proposition]{Lemma}

\theoremstyle{definition}
\newtheorem{remark}{Remark}

\allowdisplaybreaks

\author{Patrik L.\ Ferrari\thanks{Institute for Applied Mathematics, Bonn University, Endenicher Allee 60, 53115 Bonn,
Germany. E-mail: {\tt ferrari@uni-bonn.de}} \and
B\'alint Vet\H o\thanks{Institute for Applied Mathematics, Bonn University, Endenicher Allee 60, 53115 Bonn,
Germany. E-mail: {\tt vetob@uni-bonn.de}}}

\title{Non-colliding Brownian bridges\\ and the asymmetric tacnode process}

\begin{document}

\maketitle

\sloppy

\begin{abstract}
We consider non-colliding Brownian bridges starting from two points and returning to the same position.
These positions are chosen such that, in the limit of large number of bridges, the two families of bridges just touch each other forming a tacnode.
We obtain the limiting process at the tacnode, the \emph{(asymmetric) tacnode process}.
It is a determinantal point process with correlation kernel given by two parameters:
(1) the curvature's ratio $\lambda>0$ of the limit shapes of the two families of bridges,
(2) a parameter $\sigma\in\R$ controlling the interaction on the fluctuation scale.
This generalizes the result for the symmetric tacnode process ($\lambda=1$ case).
\end{abstract}

\section{Introduction and results}

Systems of non-colliding Brownian motions have been much studied recently. They arise in random matrix theory
(see e.g.~\cite{FN98,KT04,KT07}), as limit processes of random walk, discrete growth models, and random tiling problems, see e.g.~\cite{KM59,Jo02b,Jo03b,Jo04,PS02,SI03,FS03,OR01,Fer04}.

Considering non-colliding Brownian bridges (as well as discrete analogues), various kinds of determinantal processes appear naturally.
Assume that the starting and ending points are chosen such that in the limit of large number of bridges occupy a region bordered by a deterministic limit shape
(see Figure~\ref{fig:tacnode} for an illustration). Then, inside the limit shape (in the bulk) one observes the process with the sine kernel, see e.g.~\cite{OR01}.
At the edge of the limit shape, the last bridge is described asymptotically by the Airy$_2$ process~\cite{PS02,FS03,Jo03b}.
Whenever there is a cusp in the limit shape, then the process around the cusp is the Pearcey process~\cite{TW06,OR07,BK09,BD10}.
All these process are quite robust, in the sense that by moving the initial and/or ending points of the bridges, the only changes are geometric
(e.g., the position and direction of the edge/cusp changes and numerical coefficients in the scaling) but the processes are the same without free parameter.

The case of the tacnode is more delicate and the limit process is described by two parameter.
Recently, three different approaches have been used to unravel the tacnode process.
In the first work, Adler, Ferrari and van Moerbeke~\cite{FAvM10} derived the symmetric tacnode process from a limit of non-intersecting random walks.
Meanwhile two other groups were after a solution for the Brownian bridge setting.
Soon after~\cite{FAvM10}, a solution appeared in term of a $4\times 4$ Riemann-Hilbert problem by Delvaux, Kuijlaars and Zhang~\cite{DKZ10}.
Their solution is for the generic tacnode process.
The third approach, leading to, in our opinion, the simplest of the three formulations was posted more recently by Johansson~\cite{Joh10}.
In the latter the asymptotic analysis was restricted to the symmetric tacnode.
In the present paper, we analyse the general case starting with the result on two sets of Brownian bridges of~\cite{Joh10}.

The equivalence between the last two formulations follows from the fact that the starting model is the same.
However it seems hard to compare the analytic formulas directly.
The equivalence of the results between the random walk and Brownian bridge case is expected by universality,
and it can be indirectly checked by analysing a discrete model with the two approaches, as it was made very recently for the double Aztec diamond in~\cite{AJvM11}.

Now we introduce the model and state the result of this paper.
We consider $(1+\lambda)n$ non-colliding standard Brownian motions with two starting points and two endpoints where $\lambda>0$ is a fixed parameter.
More precisely, $n$ of the Brownian motions start at $a_1$ at time $0$ and arrive at $a_1$ at time $1$,
the remaining\footnote{We do not write here integer part of $\lambda n$ to keep the notation simple.}
$\lambda n$ Brownian particles have starting and ending points at $a_2$ at time $0$ and $1$ respectively with $a_1<a_2$.
For finite times $t_1,\dots,t_k$, the positions of the particles at these times form an extended determinantal point process
(for more informations on determinantal point processes, see~\cite{Lyo03,BKPV05,Sos06,Jo05,Spo05}).
For a fixed integer $n$ and a fixed $\lambda>0$, let us denote by $\L_{n,\lambda n}(s,u,t,v)$ the kernel of this determinantal point process with $s,t\in(0,1)$ and $u,v\in\R$.
The kernel $\L_{n,\lambda n}(s,u,t,v)$ was obtained in~\cite{Joh10}, see Theorem~\ref{thm:joh} below for the formula.

\begin{figure}
\begin{center}
\includegraphics[height=7.5cm]{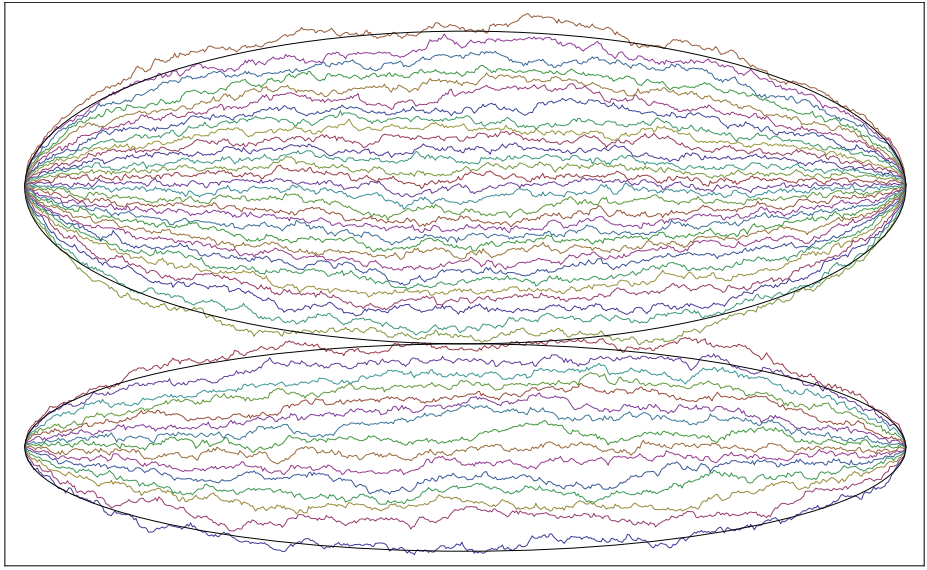}
\caption{The asymmetric system of non-colliding Brownian motions with $15$ respectively $30$ paths in the two groups, i.e.\ $n=15,\lambda=2$.}
\label{fig:tacnode}
\end{center}
\end{figure}

In this paper, we take the $n\to\infty$ limit in the model described above.
The global picture is that the two systems of non-colliding Brownian motions form two ellipses touching each other at a tacnode (see Figure~\ref{fig:tacnode} for an illustration).
Under proper rescaling, we obtain a limiting determinantal point process in the neighborhood of the point of tangency.

Here we consider the general case when two parameters modulate the limit process.
One of them is the strength of interaction, called $\sigma$, the other one is a measure of asymmetry,
called $\lambda$ which we have chosen to be the ratio of curvatures of the two ellipses at the point of tangency. For $\lambda=1$, we get back to the symmetric case treated in~\cite{Joh10}.

The scaling of the starting and ending points is
\begin{align}
a_1&=-\left(\sqrt n+\frac\sigma2n^{-1/6}\right),\label{defa1}\\
a_2&=\sqrt\lambda\left(\sqrt n+\frac\sigma2n^{-1/6}\right).\label{defa2}
\end{align}
We denote by $a$ the distance of the two endpoints:
\begin{equation}
a=a_2-a_1=\left(1+\sqrt\lambda\right)\left(\sqrt n+\frac\sigma2n^{-1/6}\right).\label{defa}
\end{equation}
In this setting, the tacnode is at $(1/2,0)$, so that the space-time scaling we need to consider is
\begin{equation}\begin{aligned}
s&=\frac12\left(1+\tau_1n^{-1/3}\right),&t&=\frac12\left(1+\tau_2n^{-1/3}\right),\\
u&=\frac12\xi_1n^{-1/6},&v&=\frac12\xi_2n^{-1/6}.\label{defus}
\end{aligned}\end{equation}

The limiting kernel takes the form
\begin{multline}
\L_{\tac}^{\lambda,\sigma}(\tau_1,\xi_1,\tau_2,\xi_2)=-\Id(\tau_1<\tau_2)p(\tau_2-\tau_1;\xi_1,\xi_2)\\
+L_{\tac}^{\lambda,\sigma}(\tau_1,\xi_1,\tau_2,\xi_2)
+\lambda^{1/6}L_{\tac}^{\lambda^{-1},\lambda^{2/3}\sigma}\left(\lambda^{1/3}\tau_1,-\lambda^{1/6}\xi_1,\lambda^{1/3}\tau_2,-\lambda^{1/6}\xi_2\right)
\label{defLL}
\end{multline}
where
\begin{equation}
p(t;x,y)=\frac1{\sqrt{4\pi t}}\exp\left(-\frac{(y-x)^2}{4t}\right)
\end{equation}
is the Gaussian kernel. To describe $L_{\tac}$, we need to introduce some notations. For a parameter $s$, let
\begin{equation}
\Ai^{(s)}(x)=e^{\frac23 s^3+xs}\Ai(s^2+x)
\end{equation}
be the extended Airy function where $\Ai^{(0)}=\Ai$ is the standard Airy function. The extended Airy kernel is given by
\begin{equation}
K_{\Ai}^{(\alpha,\beta)}(x,y)=\int_0^\infty \Ai^{(\alpha)}(x+u)\Ai^{(\beta)}(y+u)\,\d u
\end{equation}
where $K_{\Ai}^{(0,0)}=K_{\Ai}$ is the standard Airy kernel.
Let us denote the function
\begin{equation}
B_{\tau,\xi}^\lambda(x)=\int_0^\infty \Ai^{(\tau)}\left(\xi+\left(1+\sqrt{\lambda^{-1}}\right)^{1/3}\mu\right)\Ai(x+\mu)\d\mu
\end{equation}
which is reminiscent of the definition of the Airy kernel.
Let also
\begin{equation}
b_{\tau,\xi}^\lambda(x)=\lambda^{1/6}\Ai^{(\lambda^{1/3}\tau)}(-\lambda^{1/6}\xi+(1+\sqrt\lambda)^{1/3}x),
\end{equation}
and define
\begin{multline}
L_\tac^{\lambda,\sigma}(\tau_1,\xi_1,\tau_2,\xi_2)=K_{\Ai}^{(-\tau_1,\tau_2)}(\sigma+\xi_1,\sigma+\xi_2)\\
+(1+\sqrt{\lambda^{-1}})^{1/3}\left\langle B_{\tau_2,\sigma+\xi_2}^\lambda-b_{\tau_2,\sigma+\xi_2}^\lambda,
\left(\Id-\chi_{\wt\sigma}K_{\Ai}\chi_{\wt\sigma}\right)^{-1}B_{-\tau_1,\sigma+\xi_1}^\lambda\right\rangle_{L^2((\wt\sigma,\infty))}
\label{defLtac}\end{multline}
where $\chi_a(x)=\Id(x>a)$ and
\begin{equation}\label{defsigmatilde}
\wt\sigma=\lambda^{1/6}(1+\sqrt\lambda)^{2/3}\sigma.
\end{equation}

Now we can state our main result.
\begin{theorem}\label{thm:main}
The (asymmetric) tacnode process $\T^{\sigma,\lambda}$ obtained by the limit of the two non-colliding families
of $n$ respectively $\lambda n$ Brownian motions under the scaling \eqref{defa1}--\eqref{defus}
in the neighborhood of the tacnode is given by the following gap probabilities.
For any $k$ and $t_1,\dots,t_k\in(0,1)$ and for any compact set $E\subset\{t_1,\dots,t_k\}\times\R$,
\begin{equation}
\P(\T^{\sigma,\lambda}(\Id_E)=\emptyset)=\det(\Id-\L_{\tac}^{\lambda,\sigma})_{L^2(E)}
\end{equation}
where $\L_{\tac}^{\lambda,\sigma}$ is the extended kernel given by \eqref{defLL}.
\end{theorem}

\begin{remark}\label{rem:contsymm}
The (asymmetric) tacnode process has an intrinsic symmetry under the reflection on the horizontal axis
that is inherited from the finite system of Brownian motions.
This corresponds to the following transformation of the variables:
\begin{align}
\lambda&\lto\lambda^{-1},\\
n&\lto\lambda n,\\
\tau_i&\lto\lambda^{1/3}\tau_i,\label{tauchange}\\
\xi_i&\lto-\lambda^{1/6}\xi_i,\label{xichange}\\
\sigma&\lto\lambda^{2/3}\sigma.\label{sigmachange}
\end{align}
The different powers of $\lambda$ in the change of parameters \eqref{tauchange}--\eqref{sigmachange}
is necessary for observing the process on the same scale.
Note that $\wt\sigma$ given in \eqref{defsigmatilde} is left invariant under the above transformation.
\end{remark}
Next we present an alternative formulation of $\L_{\tac}^{\lambda,\sigma}$, inspired from the analogue reformulation of the kernel in~\cite{AJvM11}. Let us introduce the function
\begin{equation}\label{defC}\begin{aligned}
C^\lambda_{\tau_,\xi}(x)&=b^\lambda_{\tau,\xi}(x)-B^\lambda_{\tau,\xi}(x)\\
&=\lambda^{1/6}\Ai^{(\lambda^{1/3}\tau)}(-\lambda^{1/6}\xi+(1+\sqrt\lambda)^{1/3}x)\\
&\qquad-\int_0^\infty \Ai^{(\tau)}(\xi+(1+\sqrt{\lambda^{-1}})^{1/3}\mu)\Ai(x+\mu)\d\mu
\end{aligned}\end{equation}
Using this definition, we can give another expression for the kernel $\L_{\tac}^{\lambda,\sigma}$ which is formally similar to \eqref{defLL},
but the ingredients can be given by a single integral as follows.

\begin{proposition}\label{prop:alternative}
With
\begin{equation}
\wt L_{\tac}^{\lambda,\sigma}(\tau_1,\xi_1,\tau_2,\xi_2)
=(1+\sqrt{\lambda^{-1}})^{1/3}
\int_{\wt\sigma}^\infty (\Id-\chi_{\wt\sigma}K_{\Ai}\chi_{\wt\sigma})^{-1}
C^\lambda_{-\tau_1,\sigma+\xi_1}(x) b^\lambda_{\tau_2,\sigma+\xi_2}(x)\,\d x,
\end{equation}
we have
\begin{multline}
\L_{\tac}^{\lambda,\sigma}(\tau_1,\xi_1,\tau_2,\xi_2)=-\Id(\tau_1<\tau_2)p(\tau_2-\tau_1;\xi_1,\xi_2)\\
+\wt L_{\tac}^{\lambda,\sigma}(\tau_1,\xi_1,\tau_2,\xi_2)
+\lambda^{1/6}\wt L_{\tac}^{\lambda^{-1},\lambda^{2/3}\sigma}\left(\lambda^{1/3}\tau_1,-\lambda^{1/6}\xi_1,\lambda^{1/3}\tau_2,-\lambda^{1/6}\xi_2\right).
\label{defLL2}
\end{multline}
\end{proposition}

\subsubsection*{Acknowledgements}
The authors would like to thanks M.~Adler, P.~van Moerbeke, and K.~Johansson for discussions regarding their work on the double Aztec diamond.
This work is supported by the Hausdorff Center for Mathematics and the German Research Foundation via the SFB611--A12 project.

\section{Johansson's formula}

In this section, we recall Theorem 1.4 of Johansson in~\cite{Joh10}, the starting point for our analysis.
He obtains a formula for the correlation kernel of two non-colliding families of Brownian particles with the following properties.
The first family consists of $n$ particles, and they start at position $a_1$ at time $0$ and end at $a_1$
at time $1$. The other family has $m$ Brownian motions which start at position $a_2$ at time $0$ and end
at position $a_2$ at time $1$ with $a_2>a_1$. This system of Brownian motions conditioned on no intersection
in the time interval $(0,1)$ forms an extended determinantal point process with kernel $\L_{n,m}(s,u,t,v)$
given as follows.

Let $a=a_2-a_1$ and $d>0$ a parameter which can be chosen freely in Theorem~\ref{thm:joh}. We use the notation
\begin{align}
\A_{s,u,t,v}^1&=\frac{d^2}{(2\pi i)^2\sqrt{(1-s)(1-t)}} \int_{i\R}\d w\int_{D_{a_1}}\d z
\left(\frac{1-w/a_1}{1-z/a_1}\right)^n\frac1{w-z}\label{defA1}\\
&\qquad\times\exp\left(-\frac{sz^2}{2(1-s)}-a_1z+\frac{uz}{1-s}+\frac{tw^2}{2(1-t)}+a_1w-\frac{vw}{1-t}\right),\notag\\
\B_{t,v}^1(x)&=\frac{d\sqrt a}{(2\pi i)^2\sqrt{1-t}} \int_{i\R}\d w\int_{D_{a_1}}\d z
\left(\frac{1-w/a_1}{1-z/a_1}\right)^n\left(1-z/a_2\right)^m\frac1{z-w}\label{defB1}\\
&\qquad\times\exp\left(\frac{tw^2}{2(1-t)}+a_1w-\frac{vw}{1-t}+axz\right),\notag\\
\beta_{t,v}^1(x)&=\frac{d\sqrt a}{2\pi i\sqrt{1-t}} \int_{i\R}\d w \left(1-\frac w{a_2}\right)^m
\exp\left(\frac{tw^2}{2(1-t)}+a_1w-\frac{vw}{1-t}+axw\right)\label{defbeta1},\\
\C_{s,u}^1(y)&=\frac{d\sqrt a}{(2\pi i)^2\sqrt{1-s}} \int_{D_{a_1}}\d z\int_{D_{a_2}}\d w
\left(\frac{1-w/a_1}{1-z/a_1}\right)^n\left(\frac1{1-w/a_2}\right)^m\label{defC1}\frac1{w-z}\\
&\qquad\times\exp\left(-\frac{sz^2}{2(1-s)}-a_1z+\frac{uz}{1-s}-ayw\right),\notag\\
M_0^1(x,y)&=\frac a{(2\pi i)^2} \int_{D_{a_1}}\d z\int_{D_{a_2}}\d w
\left(\frac{1-w/a_1}{1-z/a_1}\right)^n\left(\frac{1-z/a_2}{1-w/a_2}\right)^m\frac1{z-w} e^{axz-ayw}\label{defM01}
\end{align}
where $D_{a_1}$ and $D_{a_2}$ are counterclockwise oriented circles around $a_1$ and $a_2$ respectively with small radii.

Very similarly, let
\begin{align}
\A_{s,u,t,v}^2&=\frac{d^2}{(2\pi i)^2\sqrt{(1-s)(1-t)}} \int_{i\R}\d w\int_{D_{a_2}}\d z
\left(\frac{1-w/a_2}{1-z/a_2}\right)^m\frac1{w-z}\\
&\qquad\times\exp\left(-\frac{sz^2}{2(1-s)}-a_2z+\frac{uz}{1-s}+\frac{tw^2}{2(1-t)}+a_2w-\frac{vw}{1-t}\right),\notag\\
\B_{t,v}^2(x)&=\frac{d\sqrt a}{(2\pi i)^2\sqrt{1-t}} \int_{i\R}\d w\int_{D_{a_2}}\d z
\left(1-z/a_1\right)^n\left(\frac{1-w/a_2}{1-z/a_2}\right)^m\frac1{z-w}\\
&\qquad\times\exp\left(\frac{tw^2}{2(1-t)}+a_2w-\frac{vw}{1-t}-axz\right),\notag\\
\beta_{t,v}^2(x)&=\frac{d\sqrt a}{2\pi i\sqrt{1-t}} \int_{i\R}\d w \left(1-\frac w{a_1}\right)^n
\exp\left(\frac{tw^2}{2(1-t)}+a_2w-\frac{vw}{1-t}-axw\right),\\
\C_{s,u}^2(y)&=\frac{d\sqrt a}{(2\pi i)^2\sqrt{1-s}} \int_{D_{a_2}}\d z\int_{D_{a_1}}\d w
\left(\frac1{1-w/a_1}\right)^n\left(\frac{1-w/a_2}{1-z/a_2}\right)^m\frac1{z-w}\\
&\qquad\times\exp\left(-\frac{sz^2}{2(1-s)}-a_2z+\frac{uz}{1-s}+ayw\right),\notag\\
M_0^2(x,y)&=\frac a{(2\pi i)^2} \int_{D_{a_2}}\d z\int_{D_{a_1}}\d w
\left(\frac{1-z/a_1}{1-w/a_1}\right)^n\left(\frac{1-w/a_2}{1-z/a_2}\right)^m\frac1{w-z} e^{-axz+ayw}.
\end{align}
Furthermore, let
\begin{equation}
q(s,u,t,v)=\frac1{\sqrt{2\pi(t-s)}}\exp\left(-\frac{(u-v)^2}{2(t-s)}+\frac{u^2}{2(1-s)}-\frac{v^2}{2(1-t)}\right)\Id(s<t)
\end{equation}
denote the conjugated Brownian kernel.

The theorem below is a consequence of Theorem 1.4 in~\cite{Joh10}.

\begin{theorem}[Johansson 2011]\label{thm:joh}
The extended determinantal kernel for $n+m$ non-colliding Brownian motions with two starting points and two endpoints described above can be written as
\begin{equation}\begin{aligned}
\L_{n,m}(s,u,t,v)=&-q(s,u,t,v)\\
&+d^{-2}\A_{s,u,t,v}^1+d^{-2}\langle(\B_{t,v}^1+\beta_{t,v}^1),(\Id-M_0^1)^{-1}\C_{s,u}^1\rangle_{L^2((1,\infty))}\\
&+d^{-2}\A_{s,u,t,v}^2+d^{-2}\langle(\B_{t,v}^2+\beta_{t,v}^2),(\Id-M_0^2)^{-1}\C_{s,u}^2\rangle_{L^2((1,\infty))}.
\end{aligned}\label{joh}\end{equation}
\end{theorem}

\begin{remark}
It is proved in Lemma 1.2 of~\cite{Joh10} that $\det(\Id-M_0^1)_{L^2((1,\infty))}>0$ and \mbox{$\det(\Id-M_0^2)_{L^2((1,\infty))}>0$} so that \eqref{joh} makes sense.
\end{remark}

\begin{remark}\label{rem:discrsymm}
This configuration has a natural symmetry. By reflecting the vertical direction, one observes the same process with parameters modified as follows:
\begin{equation}
n\leftrightarrow m, \quad a_1\to-a_2, \quad a_2\to-a_1,\quad u\to-u, \quad v\to-v,
\label{discrtransform}\end{equation}
$s$ and $t$ are unchanged. It is easy to check that the ingredients $\A_{s,u,t,v}^1, \B_{t,v}^1, \beta_{t,v}^1, \C_{s,u}^1$ and $M_0^1$
of the kernel of the finite system in \eqref{defA1}--\eqref{defM01} transform to their counterparts with upper index $2$
after taking the change of parameters \eqref{discrtransform} and that $q(s,u,t,v)$ is invariant under this action.

The symmetry of the limiting tacnode process established in Remark~\ref{rem:contsymm} is a consequence of the discrete symmetry.
\end{remark}

\section{Proof of the main results}
First, we give the proof of Theorem~\ref{thm:main} using two lemmas which are proved in Section~\ref{s:asymptanal}.
We start with formula \eqref{joh}, and we apply it to our present setting.
Then, we perform asymptotic analysis for the functions obtained in this way.

The appropriate order of the parameter $d$ is
\begin{equation}
d=\frac{n^{-1/12}}{\sqrt2},\label{defd}
\end{equation}
since we want $d^2\L_{n,\lambda n}$ to converge, so $d^2$ is the scaling of the space variables, see \eqref{defus}.

The strategy of the proof is that first, we establish pointwise convergence of the elements of the kernel to the appropriate functions in Lemma~\ref{lem:ptwconv}.
Then, we give uniform bounds on the functions in Lemma~\ref{lem:bounds}.
This gives, using dominated convergence, that the kernel $\L_{n,\lambda n}$ under the scaling \eqref{defa1}--\eqref{defus} converges pointwise to $\L_{\tac}^{\sigma,\lambda}$.
It turns out that this convergence is uniform on compact sets. Dominated convergence ensures that also the gap probabilities expressed by Fredholm determinants converge.

\begin{lemma}\label{lem:ptwconv}
Under the scaling given by \eqref{defa1}--\eqref{defus}, the following pointwise limits hold as $n\to\infty$.
\begin{align}
\A_{s,u,t,v}^1&\to K_{\Ai}^{(-\tau_1,\tau_2)}(\sigma+\xi_1,\sigma+\xi_2),\label{A1conv}\\
n^{-1/3}\B_{t,v}^1(1+xn^{-2/3})&\to-(1+\sqrt\lambda)^{1/2} B_{\tau_2,\sigma+\xi_2}^\lambda(\wt\sigma+\wt x),\label{B1conv}\\
n^{-1/3}\beta_{t,v}^1(1+xn^{-2/3})&\to(1+\sqrt\lambda)^{1/2} b_{\tau_2,\sigma+\xi_2}^\lambda(\wt\sigma+\wt x),\\
n^{-1/3}\C_{s,u}^1(1+yn^{-2/3})&\to-(1+\sqrt\lambda)^{1/2} B_{-\tau_1,\sigma+\xi_1}^\lambda(\wt\sigma+\wt y),\\
n^{-2/3}M_0^1(1+xn^{-2/3},1+yn^{-2/3})&\to\lambda^{1/6}(1+\sqrt\lambda)^{2/3}K_{\Ai}(\wt\sigma+\wt x,\wt\sigma+\wt y)\label{M01conv}.
\end{align}
where $\wt x=\lambda^{1/6}(1+\sqrt\lambda)^{2/3} x$ and $\wt y=\lambda^{1/6}(1+\sqrt\lambda)^{2/3} y$ and see \eqref{defsigmatilde}.

The analogue for the second set of terms is
\begin{align}
\A_{s,u,t,v}^2&\to\lambda^{1/6} K_{\Ai}^{(-\lambda^{1/3}\tau_1,\lambda^{1/3}\tau_2)}(\lambda^{2/3}\sigma-\lambda^{1/6}\xi_1,\lambda^{2/3}\sigma-\lambda^{1/6}\xi_2),\label{A2conv}\\
n^{-1/3}\B_{t,v}^2(1+xn^{-2/3})&\to-\lambda^{1/6}(1+\sqrt\lambda)^{1/2} B_{\lambda^{1/3}\tau_2,\lambda^{2/3}\sigma-\lambda^{1/6}\xi_2}^{\lambda^{-1}}(\wt\sigma+\wt x),\\
n^{-1/3}\beta_{t,v}^2(1+xn^{-2/3})&\to\lambda^{1/6}(1+\sqrt\lambda)^{1/2} b_{\lambda^{1/3}\tau_2,\lambda^{2/3}\sigma-\lambda^{1/6}\xi_2}^{\lambda^{-1}}(\wt\sigma+\wt x),\\
n^{-1/3}\C_{s,u}^2(1+yn^{-2/3})&\to-\lambda^{1/6}(1+\sqrt\lambda)^{1/2} B_{-\lambda^{1/3}\tau_1,\lambda^{2/3}\sigma-\lambda^{1/6}\xi_1}^{\lambda^{-1}}(\wt\sigma+\wt y),\\
n^{-2/3}M_0^1(1+xn^{-2/3},1+yn^{-2/3})&\to\lambda^{1/6}(1+\sqrt\lambda)^{2/3}K_{\Ai}(\wt\sigma+\wt x,\wt\sigma+\wt y).
\label{M02conv}
\end{align}
The convergence is uniform for $\xi_1$ and $\xi_2$ in a compact subset of $\R$.
\end{lemma}

\begin{lemma}\label{lem:bounds}
There are constants $c,C>0$ such that for all $x,y\ge0$, we have the following bounds
\begin{align}
\left|\A_{s,u,t,v}^i\right|&\le C\label{Abound}\\
\left|n^{-1/3}\B_{t,v}^i(1+xn^{-2/3})\right|&\le Ce^{-cx}\label{Bbound}\\
\left|n^{-1/3}\beta_{t,v}^i(1+xn^{-2/3})\right|&\le Ce^{-cx}\label{betabound}\\
\left|n^{-1/3}\C_{s,u}^i(1+yn^{-2/3})\right|&\le Ce^{-cy}\label{Cbound}\\
\left|n^{-2/3}M_0^i(1+xn^{-2/3},1+yn^{-2/3})\right|&\le Ce^{-c(x+y)}\label{Mbound}
\end{align}
for $i=1,2$. These bounds are uniform for $\xi_1$ and $\xi_2$ in a compact subset of $\R$.
\end{lemma}

\begin{proof}[Proof of Theorem~\ref{thm:main}]
First, we show that, with the scaling \eqref{defa1}--\eqref{defus} and \eqref{defd}, we have
\begin{equation}
d^2\L_{n,\lambda n}(s,u,t,v)\lto\L_{\tac}^{\lambda,\sigma}(\tau_1,\xi_1,\tau_2,\xi_2).\label{thmrefr}
\end{equation}
The convergence of $d^2q(s,t,u,v)$ to the first term on the right-hand side of \eqref{defLL} is obvious.
By \eqref{A1conv} and \eqref{A2conv}, it is enough to work with the scalar products in \eqref{joh}.

As in the original formulation in~\cite{Joh10}, we write
\begin{equation}\label{scalarprodrew}
\langle(\B_{t,v}^1+\beta_{t,v}^1),(\Id-M_0^1)^{-1}\C_{s,u}^1\rangle_{L^2((1,\infty))}
=\frac{\det(\Id-M_0^1+(\B_{t,v}^1+\beta_{t,v}^1)\otimes\C_{s,u}^1)_{L^2((1,\infty))}}{\det(\Id-M_0^1)_{L^2((1,\infty))}}.
\end{equation}
For this proof, let
\begin{equation}\D(x,y)=M_0^1(x,y)-(\B_{t,v}^1(x)+\beta_{t,v}^1(x))C_{s,u}^1(y).\end{equation}
Then the Fredholm determinant in the numerator of \eqref{scalarprodrew} can be expressed as
\begin{equation}\sum_{m=0}^\infty \frac{(-1)^m}{m!}\int_{[1,\infty)^m} \det(\D(\rho_i,\rho_j))_{1\le i,j\le m}\,\d^m\rho.\end{equation}
This is equal to
\begin{equation}\label{DFredholm}
\sum_{m=0}^\infty \frac{(-1)^m}{m!}\int_{[0,\infty)^m}
\det\left(\frac{n^{-2/3}}{\lambda^{\frac16}(1+\sqrt\lambda)^{\frac23}}
\D\left(1+\frac{x_in^{-2/3}}{\lambda^{\frac16}(1+\sqrt\lambda)^{\frac23}},1+\frac{x_jn^{-2/3}}{\lambda^{\frac16}(1+\sqrt\lambda)^{\frac23}}\right)\right)\,\d^m x
\end{equation}
after the change of variables $\rho_i=1+x_in^{-2/3}\lambda^{-1/6}(1+\sqrt\lambda)^{-2/3}$.

Using the pointwise convergence in \eqref{B1conv}--\eqref{M01conv} along with the bounds \eqref{Bbound}--\eqref{Mbound}
and Hadamard bound on the determinant\footnote{Hadamard bound:
the absolute value of a determinant of a $n\times n$ matrix with entries of absolute value not exceeding $1$ is bounded by $n^{n/2}$.},
the dominated convergence theorem implies that the numerator of \eqref{scalarprodrew} converges to
\begin{equation}\label{Dlim}
\sum_{m=0}^\infty \frac{(-1)^m}{m!}\int_{[0,\infty)^m} \det(D(x_i,x_j))_{1\le i,j\le m}\,\d^m x
\end{equation}
with
\begin{multline}
D(x,y)=K_{\Ai}(\wt\sigma+x,\wt\sigma+y)\\
-(1+\sqrt{\lambda^{-1}})^{-1/3}\left(B_{\tau_2,\sigma+\xi_2}^\lambda(\wt\sigma+x)-b_{\tau_2,\sigma+\xi_2}^\lambda(\wt\sigma+x)\right)
B_{-\tau_1,\sigma+\xi_1}^\lambda(\wt\sigma+y).
\end{multline}
A similar argument can be used for the denominator of \eqref{scalarprodrew}, which gives that the expression in
\eqref{scalarprodrew} converges to the second term on the right-hand side of \eqref{defLtac}.

In the same way one shows that
\begin{equation}\label{scalarprod2}
\langle(\B_{t,v}^2+\beta_{t,v}^2),(\Id-M_0^2)^{-1}\C_{s,u}^2\rangle_{L^2((1,\infty))}
\end{equation}
converges to the scalar product appearing in the last term of \eqref{defLL} using the remaining set of assertions in Lemma~\ref{lem:ptwconv} and also Lemma~\ref{lem:bounds}.
Alternatively, one can refer to the symmetry established in Remark~\ref{rem:discrsymm} to get the limit of \eqref{scalarprod2}.
This verifies \eqref{thmrefr}.

It remains to argue that the process $\T^{\sigma,\lambda}$ exists as a determinantal point process. For this, we give a uniform bound on $d^2\L_{n,\lambda n}(s,u,t,v)$
as $\xi_1$ and $\xi_2$ are from a compact subset of $\R$.

For $d^2q(s,u,t,v)$, the assertion is clear, for $\A_{s,u,t,v}^i$, it follows from \eqref{Abound}. The two scalar products in \eqref{joh} are bounded as follows.
We again consider the right-hand side of \eqref{scalarprodrew}. In the Fredholm expansion of the numerator \eqref{DFredholm}, we get the bound
\begin{equation}
\left|\frac{n^{-2/3}}{\lambda^{\frac16}(1+\sqrt\lambda)^{\frac23}}
\D\left(1+\frac{x_in^{-2/3}}{\lambda^{\frac16}(1+\sqrt\lambda)^{\frac23}},1+\frac{x_jn^{-2/3}}{\lambda^{\frac16}(1+\sqrt\lambda)^{\frac23}}\right)\right|
\le Ce^{-c(x_i+x_j)}
\end{equation}
with uniform $C,c>0$ on the compact subsets based on \eqref{Bbound}--\eqref{Mbound}.
Hence the convergence of \eqref{DFredholm} to \eqref{Dlim} is uniform as $\xi_1$ and $\xi_2$ are in a compact set.
The denominator of \eqref{scalarprodrew} is strictly positive by Lemma 1.2 of~\cite{Joh10} and it converges to
\begin{equation}
\det(\Id-M_0^1)_{L^2((1,\infty))}\lto\det(\Id-\chi_{\wt\sigma}K_{\Ai}\chi_{\wt\sigma})_{L^2((\wt\sigma,\infty))}=F_2(\wt\sigma)>0
\end{equation}
where $F_2$ is the \emph{Tracy-Widom distribution} function~\cite{TW94}. This shows that the first scalar product in \eqref{joh} remains uniformly bounded on compact sets.
One can proceed similarly with the second one. Therefore, the existence of the gap probabilities of the process $\T^{\sigma,\lambda}$
follows by expanding the Fredholm determinant of the finite size kernel and from the Hadamard bound on the determinant. This completes the proof of Theorem~\ref{thm:main}.
\end{proof}

The following proof is similar to that of Theorem 1.3 in \cite{AJvM11}, but the idea is adapted to the asymmetric case, so we give it completely.

\begin{proof}[Proof of Proposition \ref{prop:alternative}]
In order to rewrite the kernel in \eqref{defLL}, we define the function
\begin{equation}\label{defS}
S^{\lambda,\sigma}_{\tau,\xi}(x)=\Ai^{(\tau)}\left(\lambda^{-1/6}(\lambda^{1/6}\xi-\lambda^{2/3}\sigma)+(1+\sqrt{\lambda^{-1}})^{1/3}x\right)
=\lambda^{1/6} b^{\lambda^{-1}}_{\lambda^{1/3}\tau,\lambda^{2/3}\sigma-\lambda^{1/6}\xi}(x)
\end{equation}
and the operator $T$ on $L^2((\wt\sigma,\infty))$ with kernel function
\begin{equation}\label{defT}
T(x,y)=\Ai(x+y-\wt\sigma).
\end{equation}
One observes that since we have
\begin{equation}
K_{\Ai}(x,y)=\int_{\wt\sigma}^\infty \Ai(x+u-\wt\sigma)\Ai(y+u-\wt\sigma)\,\d u
\end{equation}
on $L^2((\wt\sigma,\infty))$, one can write
\begin{equation}
\chi_{\wt\sigma}K_{\Ai}\chi_{\wt\sigma}=T^2
\end{equation}
and also
\begin{equation}\label{kernelsum}
(\Id-\chi_{\wt\sigma}K_{\Ai}\chi_{\wt\sigma})^{-1}=\sum_{r=0}^\infty T^{2r}.
\end{equation}

Note also that, using the notations \eqref{defS} and \eqref{defT}, we have
\begin{equation}\label{Crewrite}
C^\lambda_{\tau,\sigma+\xi}(x)=\lambda^{1/6} S^{\lambda^{-1},\lambda^{2/3}\sigma}_{\lambda^{1/3}\tau,-\lambda^{1/6}\xi}(x)-TS^{\lambda,\sigma}_{\tau,\xi}(x).
\end{equation}
Similarly,
\begin{equation}\label{Crewrite2}
C^{\lambda^{-1}}_{\lambda^{1/3}\tau,\lambda^{2/3}\sigma-\lambda^{1/6}\xi}(x)
=\lambda^{-1/6} S^{\lambda,\sigma}_{\tau,\xi}(x)-TS^{\lambda^{-1},\lambda^{2/3}\sigma}_{\lambda^{1/3}\tau,-\lambda^{1/6}\xi}(x).
\end{equation}
One can also see easily that
\begin{equation}\label{Kscalarprod}
K_{\Ai}^{(-\tau_1,\tau_2)}(\sigma+\xi_1,\sigma+\xi_2)
=(1+\sqrt{\lambda^{-1}})^{1/3}\left\langle S^{\lambda,\sigma}_{-\tau_1,\xi_1},S^{\lambda,\sigma}_{\tau_2,\xi_2}\right\rangle_{L^2((\wt\sigma,\infty))},
\end{equation}
and also
\begin{multline}\label{Kscalarprod2}
K_{\Ai}^{(-\lambda^{1/3}\tau_1,\lambda^{1/3}\tau_2)}\left(\lambda^{2/3}\sigma-\lambda^{1/6}\xi_1,\lambda^{2/3}\sigma-\lambda^{1/6}\xi_2\right)\\
=(1+\sqrt\lambda)^{1/3}\left\langle S^{\lambda^{-1},\lambda^{2/3}\sigma}_{-\lambda^{1/3}\tau_1,-\lambda^{1/6}\xi_1},
S^{\lambda^{-1},\lambda^{2/3}\sigma}_{\lambda^{1/3}\tau_2,-\lambda^{1/6}\xi_2}\right\rangle_{L^2((\wt\sigma,\infty))}.
\end{multline}

Starting from \eqref{defLL} and \eqref{defLtac}, we can rewrite the kernel $\L_{\tac}^{\lambda,\sigma}$ using \eqref{defC} as follows.
\begin{equation}\begin{aligned}
&\hspace{-1.5em}\L_{\tac}^{\lambda,\sigma}(\tau_1,\xi_1,\tau_2,\xi_2)\\
=&-\Id(\tau_1<\tau_2)p(\tau_2-\tau_1;\xi_1,\xi_2)\\
&+K_{\Ai}^{(-\tau_1,\tau_2)}(\sigma+\xi_1,\sigma+\xi_2)
+\lambda^{1/6}K_{\Ai}^{(-\lambda^{1/3}\tau_1,\lambda^{1/3}\tau_2)}(\lambda^{2/3}\sigma-\lambda^{1/6}\xi_1,\lambda^{2/3}\sigma-\lambda^{1/6}\xi_2)\\
&+(1+\sqrt{\lambda^{-1}})^{1/3} \int_{\wt\sigma}^\infty (\Id-\chi_{\wt\sigma}K_{\Ai}\chi_{\wt\sigma})^{-1} C^\lambda_{\tau_2,\sigma+\xi_2}(x)
(C^\lambda_{-\tau_1,\sigma+\xi_1}(x)-b^\lambda_{-\tau_1,\sigma+\xi_1}(x))\,\d x\\
&+\lambda^{1/6}(1+\sqrt\lambda)^{1/3} \int_{\wt\sigma}^\infty (\Id-\chi_{\wt\sigma}K_{\Ai}\chi_{\wt\sigma})^{-1}
C^{\lambda^{-1}}_{\lambda^{1/3}\tau_2,\lambda^{2/3}\sigma-\lambda^{1/6}\xi_2}(x)\\
&\qquad\times(C^{\lambda^{-1}}_{-\lambda^{1/3}\tau_1,\lambda^{2/3}\sigma-\lambda^{1/6}\xi_1}(x)-b^{\lambda^{-1}}_{-\lambda^{1/3}\tau_1,\lambda^{2/3}\sigma-\lambda^{1/6}\xi_1}(x))\,\d x.
\end{aligned}\end{equation}
If we use \eqref{Kscalarprod} and \eqref{Kscalarprod2} for the first two Airy kernels, \eqref{Crewrite}, \eqref{Crewrite2} and \eqref{defS} for the two integrals
and \eqref{kernelsum}, then we get
\begin{equation}\label{LwithS}\begin{aligned}
\L_{\tac}^{\lambda,\sigma}(\tau_1,\xi_1,\tau_2,\xi_2)=&-\Id(\tau_1<\tau_2)p(\tau_2-\tau_1;\xi_1,\xi_2)\\
&+(1+\sqrt{\lambda^{-1}})^{1/3} \left\langle\sum_{r=0}^\infty T^{2r} S^{\lambda,\sigma}_{-\tau_1,\xi_1},S^{\lambda,\sigma}_{\tau_2,\xi_2}\right\rangle_{L^2((\wt\sigma,\infty))}\\
&+\lambda^{1/6}(1+\sqrt\lambda)^{1/3} \left\langle\sum_{r=0}^\infty T^{2r} S^{\lambda^{-1},\lambda^{2/3}\sigma}_{-\lambda^{1/3}\tau_1,-\lambda^{1/6}\xi_1},
S^{\lambda^{-1},\lambda^{2/3}\sigma}_{\lambda^{1/3}\tau_2,-\lambda^{1/6}\xi_2}\right\rangle_{L^2((\wt\sigma,\infty))}\\
&-(1+\sqrt\lambda)^{1/3} \left\langle\sum_{r=0}^\infty T^{2r+1} S^{\lambda,\sigma}_{-\tau_1,\xi_1},
S^{\lambda^{-1},\lambda^{2/3}\sigma}_{\lambda^{1/3}\tau_2,-\lambda^{1/6}\xi_2}\right\rangle_{L^2((\wt\sigma,\infty))}\\
&-(1+\sqrt\lambda)^{1/3} \left\langle\sum_{r=0}^\infty T^{2r+1} S^{\lambda^{-1},\lambda^{2/3}\sigma}_{-\lambda^{1/3}\tau_1,-\lambda^{1/6}\xi_1},
S^{\lambda,\sigma}_{\tau_2,\xi_2}\right\rangle_{L^2((\wt\sigma,\infty))}.
\end{aligned}\end{equation}
Note that the scalar products in the third and the fourth terms on the right-hand side of \eqref{LwithS} can be combined using \eqref{Crewrite},
respectively, the second and the fifth terms can be joined by \eqref{Crewrite2} yielding
\begin{equation}\begin{aligned}
\L_{\tac}^{\lambda,\sigma}(\tau_1,\xi_1,\tau_2,\xi_2)
=&-\Id(\tau_1<\tau_2)p(\tau_2-\tau_1;\xi_1,\xi_2)\\
&+(1+\sqrt\lambda)^{1/3} \left\langle (\Id-\chi_{\wt\sigma}K_{\Ai}\chi_{\wt\sigma})^{-1} C^\lambda_{-\tau_1,\sigma+\xi_1},
S^{\lambda^{-1},\lambda^{2/3}\sigma}_{\lambda^{1/3}\tau_2,-\lambda^{1/6}\xi_2}\right\rangle_{L^2((\wt\sigma,\infty))}\\
&+(1+\sqrt\lambda)^{1/3} \left\langle (\Id-\chi_{\wt\sigma}K_{\Ai}\chi_{\wt\sigma})^{-1} C^{\lambda^{-1}}_{-\lambda^{1/3}\tau_1,\lambda^{2/3}\sigma-\lambda^{1/6}\xi_1},
S^{\lambda,\sigma}_{\tau_2,\xi_2}\right\rangle_{L^2((\wt\sigma,\infty))}
\end{aligned}\end{equation}
which proves the proposition by \eqref{defS}.
\end{proof}

\section{Asymptotic analysis}\label{s:asymptanal}

\begin{proof}[Proof of Lemma~\ref{lem:ptwconv}]
We write down the ingredients of the kernel in Theorem~\ref{thm:joh} with the values given by \eqref{defa1}--\eqref{defa2}.
By Remark~\ref{rem:discrsymm}, it is enough to prove the first set of statements \eqref{A1conv}--\eqref{M01conv}.
Then, we use the method of saddle point analysis to get the limits in Lemma~\ref{lem:ptwconv}.
For the asymptotic analysis we use a standard pattern, explained e.g.\ in Section 6 of~\cite{BF08}.

For \eqref{defA1}, after change of variables $w\to aw/(1+\sqrt\lambda)$ and $z\to-az/(1+\sqrt\lambda)$, we get
\begin{equation}\begin{aligned}
\A_{s,u,t,v}^1=&\frac{d^2a}{(1+\sqrt\lambda)\sqrt{(1-s)(1-t)}}\frac1{(2\pi i)^2} \int_{i\R}\d w\int_{D_1}\d z
\left(\frac{1+w}{1-z}\right)^n\frac1{w+z}\\
&\times\exp\left(-\frac{a^2sz^2}{(1+\sqrt\lambda)^2 2(1-s)}-\frac{a^2z}{(1+\sqrt\lambda)^2}-\frac{auz}{(1+\sqrt\lambda)(1-s)}\right)\\
&\times\exp\left(\frac{a^2tw^2}{(1+\sqrt\lambda)^2 2(1-t)}-\frac{a^2w}{(1+\sqrt\lambda)^2}-\frac{avw}{(1+\sqrt\lambda)(1-t)}\right).
\end{aligned}\label{A1rew1}\end{equation}
First we observe that $\Re(z+w)>0$, hence we can write
\begin{equation}\frac1{z+w}=n^{1/3}\int_0^\infty e^{-n^{1/3}\mu(z+w)}\,\d\mu.\end{equation}
By substituting this and \eqref{defa}--\eqref{defus} and by Taylor expansion, we obtain
\begin{equation}\begin{aligned}
\A_{s,u,t,v}^1=&(n^{2/3}+o(1)) \int_0^\infty\d\mu\\
&\times\frac1{2\pi i}\int_{i\R}\d w \exp\left(n\left[\log(1+w)+\frac{w^2}2-w\right]+n^{2/3}\tau_2w^2-n^{1/3}[\sigma+\xi_2+\mu]w\right)\\
&\times\frac1{2\pi i}\int_{D_1}\d z \exp\left(n\left[-\log(1-z)-\frac{z^2}2-z\right]-n^{2/3}\tau_1z^2-n^{1/3}[\sigma+\xi_1+\mu]z\right)\\
&\times\exp\left(\O\left(n^{1/3}w^2+w+n^{1/3}z^2+z\right)\right).
\end{aligned}\label{A1comp}\end{equation}

A change of variables $w\to aw/(1+\sqrt\lambda)$ and $z\to-az/(1+\sqrt\lambda)$ in \eqref{defB1} gives
\begin{equation}\begin{aligned}
\B_{t,v}^1(x)=&\frac{da^{3/2}}{(1+\sqrt\lambda)\sqrt{1-t}}\frac1{(2\pi i)^2} \int_{i\R}\d w\int_{D_1}\d z
\left(\frac{1+w}{1-z}\right)^n\left(1+\frac z{\sqrt\lambda}\right)^{\lambda n}\frac1{z+w}\\
&\times\exp\left(\frac{a^2tw^2}{(1+\sqrt\lambda)^2 2(1-t)}-\frac{a^2w}{(1+\sqrt\lambda)^2}-\frac{avw}{(1+\sqrt\lambda)(1-t)}-\frac{a^2xz}{1+\sqrt\lambda}\right)
\end{aligned}\end{equation}
which yields
\begin{equation}\begin{aligned}
&n^{-1/3}\B_{t,v}^1(1+xn^{-2/3})=(n^{2/3}+o(1))(1+\sqrt\lambda)^{1/2} \int_0^\infty\d\mu\\
&\qquad\times\frac1{2\pi i}\int_{i\R}\d w \exp\left(n\left[\log(1+w)+\frac{w^2}2-w\right]-n^{2/3}\tau_2w^2-n^{1/3}[\sigma+\xi_2+\mu]w\right)\\
&\qquad\times\frac1{2\pi i}\int_{D_1}\d z \exp\left(n\left[\lambda\log\left(1+\frac z{\sqrt\lambda}\right)-\log(1-z)-(1+\sqrt\lambda)z\right]\right)\\
&\qquad\times\exp\left(-n^{1/3}[(1+\sqrt\lambda)(x+\sigma)+\mu]z+\O\left(n^{1/3}w^2+w+n^{1/3}z^2+z\right)\right).
\end{aligned}\label{B1comp}\end{equation}

In the definition of $\beta_{t,v}^1$, after $w\to-aw/(1+\sqrt\lambda)$, one obtains
\begin{equation}\begin{aligned}
\beta_{t,v}^1(x)=&\frac{da^{3/2}}{(1+\sqrt\lambda)\sqrt{1-t}}\frac1{2\pi i} \int_{i\R}\d w \left(1+\frac w{\sqrt\lambda}\right)^{\lambda n}\\
&\times\exp\left(\frac{a^2tw^2}{(1+\sqrt\lambda)^2 2(1-t)}+\frac{a^2w}{(1+\sqrt\lambda)^2}+\frac{avw}{(1+\sqrt\lambda)(1-t)}-\frac{a^2xw}{1+\sqrt\lambda}\right).
\end{aligned}\end{equation}
Hence,
\begin{equation}\begin{aligned}
&n^{-1/3}\beta_{t,v}^1(1+xn^{-2/3})=(n^{2/3}+o(1))(1+\sqrt\lambda)^{1/2}\\
&\qquad\times\frac1{2\pi i}\int_{i\R}\d w \exp\left(n\left[\lambda\log\left(1+\frac w{\sqrt\lambda}\right)+\frac{w^2}2-w\right]\right)\\
&\qquad\times\exp\left(n^{2/3}\tau_2w^2-n^{1/3}[\sqrt\lambda\sigma-\xi_2+(1+\sqrt\lambda)x]w+\O\left(n^{1/3}w^2+w\right)\right).
\end{aligned}\label{beta1comp}\end{equation}

Similarly in \eqref{defC1} with $w\to aw/(1+\sqrt\lambda)$ and $z\to-az/(1+\sqrt\lambda)$, we get
\begin{equation}\begin{aligned}
\C_{s,u}^1(y)=&-\frac{da^{3/2}}{(1+\sqrt\lambda)\sqrt{1-s}}\frac1{(2\pi i)^2} \int_{D_1}\d z\int_{D_{\sqrt\lambda}}\d w
\left(\frac{1+w}{1-z}\right)^n\left(\frac1{1-\frac w{\sqrt\lambda}}\right)^{\lambda n}\frac1{w+z}\\
&\times\exp\left(-\frac{a^2sz^2}{(1+\sqrt\lambda)^2 2(1-s)}-\frac{a^2z}{(1+\sqrt\lambda)^2}-\frac{auz}{(1+\sqrt\lambda)(1-s)}-\frac{a^2yw}{1+\sqrt\lambda}\right).
\end{aligned}\end{equation}
After substituting \eqref{defa}--\eqref{defus}, it becomes
\begin{equation}\begin{aligned}
&n^{-1/3}\C_{s,u}^1(1+yn^{-2/3})=-(n^{2/3}+o(1))(1+\sqrt\lambda)^{1/2} \int_0^\infty\d\mu\\
&\qquad\times\frac1{2\pi i}\int_{D_1}\d z \exp\left(n\left[-\log(1-z)-\frac{z^2}2-z\right]-n^{2/3}\tau_1z^2-n^{1/3}[\sigma+\xi_1+\mu]z\right)\\
&\qquad\times\frac1{2\pi i}\int_{D_{\sqrt\lambda}}\d w \exp\left(n\left[\log(1+w)-\lambda\log\left(1-\frac w{\sqrt\lambda}\right)-(1+\sqrt\lambda)w\right]\right)\\
&\qquad\times\exp\left(-n^{1/3}[(1+\sqrt\lambda)(\sigma+y)+\mu]w+\O\left(n^{1/3}w^2+w+n^{1/3}z^2+z\right)\right).
\end{aligned}\label{C1comp}\end{equation}

Using the same change of variables $w\to aw/(1+\sqrt\lambda)$ and $z\to-az/(1+\sqrt\lambda)$ in $M_0^1$, we have
\begin{equation}\begin{aligned}
M_0^1(x,y)=&\frac{a^2}{(1+\sqrt\lambda)}\frac1{(2\pi i)^2} \int_{D_1}\d z\int_{D_{\sqrt\lambda}}\d w
\left(\frac{1+w}{1-z}\right)^n\left(\frac{1+\frac z{\sqrt\lambda}}{1-\frac w{\sqrt\lambda}}\right)^{\lambda n}\frac1{z+w}\\
&\times\exp\left(-\frac{a^2xz}{1+\sqrt\lambda}-\frac{a^2yw}{1+\sqrt\lambda}\right).
\end{aligned}\label{M01rew1}\end{equation}
That is,
\begin{equation}\begin{aligned}
&n^{-2/3}M_0^1(1+xn^{-2/3},1+yn^{-2/3})=(n^{2/3}+o(1))(1+\sqrt\lambda) \int_0^\infty\d\mu\\
&\qquad\times\frac1{2\pi i}\int_{D_{\sqrt\lambda}}\d w \exp\left(n\left[\log(1+w)-\lambda\log\left(1-\frac w{\sqrt\lambda}\right)-(1+\sqrt\lambda)w\right]\right)\\
&\qquad\times\exp\left(-n^{1/3}[(1+\sqrt\lambda)(\sigma+y)+\mu]w\right)\\
&\qquad\times\frac1{2\pi i}\int_{D_1}\d z \exp\left(n\left[\lambda\log\left(1+\frac z{\sqrt\lambda}\right)-\log(1-z)-(1+\sqrt\lambda)z\right]\right)\\
&\qquad\times\exp\left(-n^{1/3}[(1+\sqrt\lambda)(\sigma+x)+\mu]z+\O\left(n^{1/3}w^2+w+n^{1/3}z^2+z\right)\right).
\end{aligned}\label{M01comp}\end{equation}
We can take the integration paths $D_1$ and $D_{\sqrt\lambda}$ to denote here circles around $1$ and $\sqrt\lambda$ with radii $1$ and $\sqrt\lambda$, in this way, passing through $0$.

In the above formulae, integrals of form
\begin{equation}\frac1{2\pi i}\int_\gamma\d z e^{nf_0(z)+n^{2/3}f_1(z)+n^{1/3}f_2(z)}\end{equation}
appear. Hence, we can follow the steps of the proof of Lemma 6.1 in~\cite{BF08}.
It can be checked that in all the cases in \eqref{A1comp}, \eqref{B1comp}, \eqref{beta1comp}, \eqref{C1comp} and \eqref{M01comp},
$w_c=z_c=0$ is a double critical point for the functions below that appear as $f_0$. We also give the Taylor expansion for later use
\begin{equation}\begin{aligned}
\log(1+w)+\frac{w^2}2-w&=\frac{w^3}3+\O\left(w^4\right),\\
\lambda\log\left(1+\frac w{\sqrt\lambda}\right)+\frac{w^2}2-w&=\frac1{\sqrt\lambda}\frac{w^3}3+\O\left(w^4\right),\\
\log(1+w)-\lambda\log\left(1-\frac w{\sqrt\lambda}\right)-(1+\sqrt\lambda)w&=\frac{1+\sqrt\lambda}{\sqrt\lambda}\frac{w^3}3+\O\left(w^4\right),
\end{aligned}\label{f0taylor}\end{equation}
and their versions after taking the transformation $-f_0(-w)$ which preserves the leading terms on the right-hand side of \eqref{f0taylor}.

The issue of finding steep descent paths is treated in a more general setting in the proofs of Lemma~\ref{lem:bounds} and~\ref{lem:examplebound} below.
Thus the details are omitted here. As it turns out, the integration paths in \eqref{A1comp}, \eqref{B1comp}, \eqref{beta1comp}, \eqref{C1comp} and \eqref{M01comp}
are steep descent for $f_0$ in the following sense ($\gamma$ means the generic integration path):
\begin{itemize}
\item $\Re(f_0(z))$ on $\gamma$ reaches its maximum at $0$,
\item $\Re(f_0(z))$ is monotone along $\gamma$ except at its maximum point $0$ and, if $\gamma=D_1$ or $D_{\sqrt\lambda}$,
then also in the antipodal points $2$ or $2\sqrt\lambda$ respectively (where the real part reaches its minimum).
\end{itemize}

In the neighborhood of $0$, we can slightly modify the paths $i\R$, $D_1$ and $D_{\sqrt\lambda}$ in such a way that they are still steep descent and that, close to $0$, the descent is the steepest.
The modification is as follows. For a small $\delta>0$ that is given precisely later, we choose the path $\{e^{-i\pi/3}t,0\le t\le\delta\}\cup\{e^{i\pi/3}t,0\le t\le\delta\}$
close to $0$. This is clearly locally steepest descent for all functions in \eqref{f0taylor}.
In the case of $i\R$, we continue with $\delta/2+i\R$ to infinity. If we had a circle, we decrease the radius slightly, and we cut off a piece of it close to $0$ in such a way that it matches
the steepest descent path. Computations which are done in the proofs of Lemma~\ref{lem:bounds} and~\ref{lem:examplebound} show that
these modified paths are still steep descent for $f_0$ in all the cases in \eqref{A1comp}, \eqref{B1comp}, \eqref{beta1comp}, \eqref{C1comp} and \eqref{M01comp}.

Along the integration paths, only the contribution on $\gamma\cap\{|w|\le\delta\}$ is considered, since, by the steep descent property of the paths, the error that is made by neglecting
the rest of the path is $\O(\exp(-c\delta^3n))$ as $n\to\infty$ uniformly as $\xi_1$ and $\xi_2$ are in bounded intervals.

We choose now $\delta$ such that the error terms in \eqref{A1comp}, \eqref{B1comp}, \eqref{beta1comp}, \eqref{C1comp} and \eqref{M01comp}
are small enough in the $\delta$ neighborhood of $0$.
That is, for the difference of the integrals with and without the error in the exponent, we apply $|e^x-1|\le|x|e^{|x|}$.
We do a change of variables $n^{1/3}w\to w$, and, by taking $\delta$ small enough, we see that the difference is $\O(n^{-1/3})$ uniformly.

Finally, we can extend the steepest descent path to $e^{-i\pi/3}\infty$ and $e^{i\pi/3}\infty$ on the price of a uniformly $\O(\exp(-c\delta^3n))$ error again.
Then, we apply the formula
\begin{equation}\begin{aligned}
\frac1{2\pi i}\int_{e^{-i\pi/3}\infty}^{e^{i\pi/3}\infty}\exp\left(a\frac{z^3}3+bz^2+cz\right)\d z
&=a^{-1/3}\exp\left(\frac{2b^3}{3a^2}-\frac{bc}a\right)\Ai\left(\frac{b^2}{a^{4/3}}-\frac c{a^{1/3}}\right)\\
&=a^{-1/3}\Ai^{(a^{-2/3}b)}\left(-\frac c{a^{1/3}}\right),
\end{aligned}\end{equation}
and we exactly get the limits \eqref{A1conv}--\eqref{M01conv}.
From this, along with the use of symmetry given in Remark~\ref{rem:discrsymm}, \eqref{A2conv}--\eqref{M02conv} follow immediately.
All the error terms can be bounded uniformly in $\xi_1$ and $\xi_2$ if they are in a compact interval, hence Lemma~\ref{lem:ptwconv} is proved.
\end{proof}

First, we give the following lemma with its full proof. In the proof of Lemma~\ref{lem:bounds}, we will use the assertion, and some similar
statements which will not be spelled out later, because they can be shown as Lemma~\ref{lem:examplebound}.

\begin{lemma}\label{lem:examplebound}
There are constants $C,c>0$ such that
\begin{equation}\label{intbound}
\left|\frac{n^{1/3}}{2\pi i}\int_{i\R}\d w\exp\left(n\left[\log(1+w)+\frac{w^2}2-w\right]+\kappa_1n^{2/3}w^2-sn^{1/3}w\right)\right|\le Ce^{-cs}
\end{equation}
for $s>0$ large enough.
\end{lemma}

\begin{proof}[Proof of Lemma~\ref{lem:examplebound}]
We follow the lines of the proof of Proposition 5.3 in~\cite{BF07}. Since we are interested in large values of $s$, we take
\begin{equation}\label{defwts}
\wt s=n^{-2/3}s,
\end{equation}
and we define
\begin{equation}
\wt f_0(w)=\log(1+w)+\frac{w^2}2-w-\wt sw.
\end{equation}
For small values of $\wt s$, this function has two critical points at $\pm\wt s^{1/2}$ at first order, and we will pass through the positive one. Hence define
\begin{equation}\label{defalpha}
\alpha=\left\{\begin{array}{ll} \wt s^{1/2} & \mbox{if}\ \wt s\le\varepsilon,\\ \varepsilon^{1/2} & \mbox{if}\ \wt s>\varepsilon,\end{array}\right.
\end{equation}
for some small $\varepsilon>0$ to be chosen later and consider the path $\Gamma=\alpha+i\R$. By the Cauchy theorem, the integral in \eqref{intbound}
does not change if we modify the integration path to $\Gamma$.

The path $\Gamma$ is steep descent for the function $\Re(\wt f_0(w))$, since
\begin{equation}
\frac{\d}{\d t}\Re(\wt f_0(\alpha+it))=-t\underbrace{\left(1-\frac1{(1+\alpha)^2+t^2}\right)}_{\geq 0\textrm{ for }\alpha,t>0.}.
\end{equation}
Define
\begin{equation}Q(\alpha)=\exp\left(\Re\left(n\wt f_0(\alpha)+n^{2/3}\kappa_1\alpha^2\right)\right).\end{equation}
Let $\Gamma_\delta=\{\alpha+it,|t|\le\delta\}$. By the steep descent property of $\Gamma$, the contribution of the integral over $\Gamma\setminus\Gamma_\delta$
in \eqref{intbound} is bounded by $Q(\alpha)\O(e^{-cn})$ where $c>0$ does not depend on $n$. The integral on $\Gamma_\delta$ can be bounded by
\begin{equation}\label{maincontr}
Q(\alpha)\left|\frac{n^{1/3}}{2\pi i}\int_{\Gamma_\delta}\d w\exp\left(n(\wt f_0(w)-\wt f_0(\alpha))+n^{2/3}\kappa_1(w^2-\alpha^2)\right)\right|.
\end{equation}
By series expansion,
\begin{equation}\Re(\wt f_0(\alpha+it)-\wt f_0(\alpha))=-\gamma t^2(1+\O(t))\end{equation}
with
\begin{equation}\label{defgamma}
\gamma=\frac12\left(1-\frac1{(1+\alpha)^2}\right).
\end{equation}
After a change of variable $w=\alpha+it$, \eqref{maincontr} is written as
\begin{equation}
\begin{aligned}
&Q(\alpha)\frac{n^{1/3}}{2\pi}\int_{-\delta}^\delta \d t\exp\left(-\gamma t^2 n(1+\O(t))\left(1+\O\left(n^{-1/3}\right)\right)\right)\\
&\le Q(\alpha)\frac{n^{1/3}}{2\pi}\int_{-\delta}^\delta \d t\exp\left(-\frac{\gamma t^2 n}2\right)\le Q(\alpha)\frac1{\sqrt{2\pi\gamma n^{1/3}}}
\end{aligned}
\end{equation}
for $\delta$ small enough and $n$ large enough.
The estimate above is the largest if $\gamma$ is small. Note that, by \eqref{defgamma}, \eqref{defalpha} and \eqref{defwts},
\begin{equation}\gamma n^{1/3}\sim\alpha n^{1/3}\sim\wt s^{1/2}n^{1/3}\sim s^{1/2}\end{equation}
which is large if $s$ is large enough, so the integral in \eqref{maincontr} is at most constant times $Q(\alpha)$.

Hence, it remains to bound $Q(\alpha)$ exponentially in $s$.
For this end, we use the Taylor expansion
\begin{equation}\wt f_0(w)=\left(\frac{w^3}3-\wt sw\right)(1+\O(w)).\end{equation}

If $\wt s\le\varepsilon$, then
\begin{equation}
\begin{aligned}
Q(\alpha)&=\exp\left(\left(-\frac23n\wt s^{3/2}+\kappa_1n^{2/3}\wt s\right)\left(1+\O\left(\sqrt\varepsilon\right)\right)\right)\\
&=\exp\left(\left(-\frac23s^{3/2}+\kappa_1s\right)\left(1+\O\left(\sqrt\varepsilon\right)\right)\right)
\end{aligned}
\end{equation}
where the first term in the exponent dominates as $s$ is large, so this is even stronger than what we had to prove.

If $\wt s>\varepsilon$, then
\begin{equation}Q(\alpha)=\exp\left(\left(n\sqrt\varepsilon\left(\frac{\varepsilon}3-\wt s\right)+\kappa_1n^{2/3}\varepsilon\right)\left(1+\O\left(\sqrt\varepsilon\right)\right)\right)\end{equation}
where $\varepsilon/3-\wt s\le-\frac23\wt s$. Hence, the first term in the exponent is about $-\frac23\sqrt\varepsilon n^{1/3}s$ which dominates the second term
that is of order $\varepsilon n^{2/3}\sim s$. Therefore, for a given $\varepsilon>0$, $n$ can be chosen so large that
\begin{equation}Q(\alpha)\le\exp\left(-\frac13\sqrt\varepsilon n^{1/3}s\right)\end{equation}
which finishes the proof.
\end{proof}

\begin{proof}[Proof of Lemma~\ref{lem:bounds}]
By Remark~\ref{rem:discrsymm}, it is enough to prove all the bounds for $i=1$.
The assertion \eqref{Abound} can be shown as follows. Lemma~\ref{lem:examplebound} with $s=\sigma+\xi_2+\mu$ applies for the integral with respect to $w$ in \eqref{A1comp}
and provides an exponentially decaying bound in $\mu$. If we prove a similar statement for the integral with respect to $z$, then \eqref{Abound} follows for $i=1$ along with
the uniformity assertion for $\xi_1$ and $\xi_2$. We omit the details of the proof of the bound on the $z$-integral here, because they are very similar to that of Lemma~\ref{lem:examplebound},
but we give that, for the second integral in \eqref{A1comp}. We consider the function
\begin{equation}\label{defg0A}
g_0^{\A}(z)=-\log(1-z)-\frac{z^2}2-z-\wt sz
\end{equation}
and the path $\{1-\rho e^{i\phi}\}$ with $0<\rho\le1$. It is steep descent because
\begin{equation}\frac{\d}{\d\phi}\Re(g_0^{\A}(1-\rho e^{i\phi}))=-\rho\sin\phi\underbrace{(2(1-\rho\cos\phi)+\wt s)}_{\ge0}.\end{equation}

For \eqref{Bbound}, we use Lemma~\ref{lem:examplebound} and the fact that, for the function
\begin{equation}\label{defg0B}
g_0^{\B}(z)=\lambda\log\left(1+\frac z{\sqrt\lambda}\right)-\log(1-z)-(1+\sqrt\lambda)z-\wt sz,
\end{equation}
the path $\{1-\rho e^{i\phi}\}$ with $0<\rho\le1$ is steep descent. Indeed
\begin{equation}\frac{\d}{\d\phi}\Re(g_0^{\B}(1-\rho e^{i\phi}))=-\rho\sin\phi\left((1+\sqrt\lambda)\left(1-\frac{\lambda}{|\sqrt\lambda+1-\rho e^{i\phi}|^2}\right)+\wt s\right)\end{equation}
where the last factor between the outhermost parenthesis is certainly positive.

For proving \eqref{betabound}, we take the function
\begin{equation}f_0^\beta(w)=\lambda\log\left(1+\frac w{\sqrt\lambda}\right)+\frac{w^2}2-w-\wt sw.\end{equation}
The steep descent path is $\alpha+i\R$ for $\alpha\ge0$ by
\begin{equation}\frac{\d}{\d t}\Re(f_0^\beta(\alpha+it))=-t\underbrace{\left(1-\frac\lambda{(\sqrt\lambda+\alpha)^2+t^2}\right)}_{\ge0}.\end{equation}

Finally, for the function
\begin{equation}\label{deff0C}
f_0^{\C}(w)=\log(1+w)-\lambda\log\left(1-\frac w{\sqrt\lambda}\right)-(1+\sqrt\lambda)w-\wt sw,
\end{equation}
we choose the steep descent path $\{\sqrt\lambda-\rho e^{i\phi}\}$ with $0<\rho\le\sqrt\lambda$. The steep descent property is shown by
\begin{equation}\frac{\d}{\d\phi}\Re(f_0^{\C}(\sqrt\lambda-\rho e^{i\phi}))=-\rho\sin\phi\left((1+\sqrt\lambda)\left(1-\frac1{|\sqrt\lambda+1-\rho e^{i\phi}|^2}\right)+\wt s\right)\end{equation}
where the factor between the outhermost parenthesis is positive. This fact, together with the steep descent property of $g_0^{\A}(z)$ in \eqref{defg0A}, yields \eqref{Cbound}.

Using the steep descent paths for $f_0^{\C}(w)$ in \eqref{deff0C} and for $g_0^{\B}(z)$ in \eqref{defg0B},
we can mimic the proof of Lemma~\ref{lem:examplebound} to get exponential bounds on the integrals in \eqref{M01comp}. It completes the proof of Lemma~\ref{lem:bounds}.
\end{proof}


\end{document}